\documentclass[12pt]{amsart}
\usepackage{amssymb}

\begin{document}

\parskip=\smallskipamount

\newtheorem{theorem}{Theorem}[section]
\newtheorem{lemma}[theorem]{Lemma}
\newtheorem{corollary}[theorem]{Corollary}
\newtheorem{proposition}[theorem]{Proposition}

\theoremstyle{definition}
\newtheorem{definition}[theorem]{Definition}
\newtheorem{example}[theorem]{Example}
\newtheorem{xca}[theorem]{Exercise}
\newtheorem{problem}[theorem]{Problem}
\newtheorem{remark}[theorem]{Remark}
\newtheorem{question}[theorem]{Question}
\newtheorem{conjecture}[theorem]{Conjecture}

\newcommand{\B}{\mathbb{B}}
\newcommand{\C}{\mathbb{C}}
\newcommand{\D}{\mathbb{D}}
\newcommand{\N}{\mathbb{N}}
\newcommand{\Q}{\mathbb{Q}}
\newcommand{\Z}{\mathbb{Z}}
\renewcommand{\P}{\mathbb{P}}
\newcommand{\R}{\mathbb{R}}
\newcommand{\T}{\mathbb{T}}
\newcommand{\U}{\mathbb{U}}

\newcommand{\cA}{\mathcal{A}}
\newcommand{\cB}{\mathcal{B}}
\newcommand{\cR}{\mathcal{R}}
\newcommand{\cC}{\mathcal{C}}
\newcommand{\cD}{\mathcal{D}}
\newcommand{\cF}{\mathcal{F}}
\newcommand{\cH}{\mathcal{H}}
\newcommand{\cL}{\mathcal{L}}
\newcommand{\cN}{\mathcal{N}}
\newcommand{\cO}{\mathcal{O}}
\newcommand{\cS}{\mathcal{S}}
\newcommand{\cZ}{\mathcal{Z}}
\newcommand{\cP}{\mathcal{P}}
\newcommand{\cT}{\mathcal{T}}
\newcommand{\cU}{\mathcal{U}}
\newcommand{\cX}{\mathcal{X}}
\newcommand{\cW}{\mathcal{W}}

\def\di{\partial}
\def\rank{\mathrm{rank}\,}

\numberwithin{equation}{section}
\setcounter{section}{0}

\author{Dejan Kolari\v{c}}
\address{Institute of Mathematics, Physics and Mechanics,
University of Ljubljana, Jadranska 19, 1000 Ljubljana, Slovenia}
\email{dejan.kolaric@fmf.uni-lj.si}
\thanks{Work on this paper was supported by ARRS, Republic of Slovenia.}

\date{\today}

\begin{abstract}
We prove the parametric homotopy principle for holomorphic immersions of Stein manifolds into Euclidian space and the homotopy principle with approximation on holomorphically convex sets. We write an integration by parts like formula for the solution $f$ to the problem $Lf|_\Sigma=g$, where $L$ is a holomorphic vector field, semi-transversal to analytic variety $\Sigma$.
\end{abstract}

\title{Parametric H-principle for holomorphic immersions with approximation}

\maketitle

\section{Introduction and results}

A $\cC^1$ map $f:X\to Y$ between smooth manifolds $X$ and $Y$ is an immersion if $\theta=df:TX\to TY$ is injective.
It was shown by Hirsch \cite{Hi} that given an injective bundle homomorphism $\theta: TX\to TY$ there exists a
continuous homotopy of injective homomorphisms $\theta_t:TX\to TY$ with $\theta_0=\theta$ and $\theta_1=df$ for
some immersion $f: X\to Y$. Following \cite{GR1} we say that smooth immersions $X\to Y$ satisfy h-principle.

This result cannot be generalized to holomorphic maps between arbitrary complex manifolds. For example there can be a lot
of injective bundle maps $TX\to T\C^N$, where $X$ is a compact complex manifold, but there are no holomorphic immersions $X\to \C^N$ when $\dim X\geq 1$, since the only compact analytic sets in $\C^N$ are finite.
Additionaly, if $X=\C^n$, there are examples of Kobayashi hyperbolic manifolds $Y$ such that every holomorphic map $\C^n\to Y$ has rank strictly less than $n$ everywhere.
However, immersions of Stein manifolds into $\C^N$ satisfy the following holomorphic homotopy principle by Eliashberg-Gromov (\textsection{2.1.5}, \cite{GR1}):

{\it Suppose that the cotangent bundle $T^\ast X$ of a Stein manifold $X$ with $\dim X=n$ is generated
by $q>n$ holomorphic $(1,0)$-forms $\varphi_1,\ldots,\varphi_q$. Then $q$-tuple
$\varphi=(\varphi_1,\ldots,\varphi_q)$ can be changed (through
$q$-tuples generating $T^\ast X$) to $(df_1,\ldots, df_q)$ where $(f_1,\ldots, f_q):X\to \C^q$ is a holomorphic immersion.}

The sketch of the proof is included in \cite{GR1}, however a lot of technical details are missing. The main idea described in \cite{GR1} is using a solution to the problem $Lf|_\Sigma=g$ to replace forms one by one by differentials. Suppose $\varphi_q$ is being replaced by a differential. Then $\Sigma\subset X$ would be an analytic set where forms $\varphi_1,\ldots,\varphi_{q-1}$ fail to span $T^\ast X$, $L:X\to TX$ would be a holomorphic vector field with  $\varphi_1 L=0,\ldots,\varphi_{q-1} L=0$ on $\Sigma$,
 $g=\varphi_q L$ a holomorphic function and $f$ an unknown holomorphic function. Let $\varphi_{q,t}=(1-t)\varphi_q+t df$.
The desired homotopy through $q$-tuples spanning $T^\ast X$ would be then $(\varphi_1,\ldots,\varphi_{q-1},\varphi_{q,t})$ since $\varphi_{q,t} L=\varphi_q L\neq 0$ on $\Sigma$.

To solve the above problem the semi-transversality condition $L\pitchfork_{\text semi} \Sigma$ (see definition \ref{def_semitrans}) introduced in \cite{GR1} is needed.
In section 4 we show in detail how to obtain this condition for $\cL$ and $\Sigma'$ which appear in
the proof of lemma $A''_4$ in \cite{GR1}, p. 69.
We use the jet transversality theorem for holomorphic maps to reduce the set of
non semi-transversal points by obtaining semi-transversality on the complements of analytic sets
and making sure we do not get any new non semi-transversal points (see lemma \ref{nontrans_reduct}).

We additionaly prove the homotopy principle with approximation (theorem \ref{main}) by writing down a formula for the solution $f$ of $Lf|_\Sigma=g$ (theorem \ref{solve_lf_a}) and observing that given a holomorphically convex compact set $K\subset X$ we get
 $|f|_K <M |g|_K$, where $M$ depends only on the restriction of data $L, \Sigma$ on $K$. Here we additionally allow small perturbations of $\Sigma$ and $L$.

We show how the  proof of theorem \ref{main} can be generalized to families of initial (1,0)-forms and use
the formula (see theorem \ref{solve_lf_a_param}) to establish the parametric homotopy principle (theorem \ref{main_parametric}). As a corollary of this and the Oka principle for sections of bundles we show (corollary \ref{homot_immers}) that if differentials of two immersions $f_1,f_2: X\to \C^q$  are homotopic through $q$-tuples of linearly independent forms, then immersions $f_1$ and $f_2$ are homotopic through immersions.

\begin{theorem} \label{main}
Let $X$ be a Stein manifold and let $q>n=\dim X\geq 1$. Let
$\varphi=(\varphi_1,\ldots,\varphi_q):X\to (T^\ast X)^q$ be
$(1,0)$-forms on Stein manifold $X$ which are holomorphic on a compact,
holomorphically convex set $K\subset X$.
Let $\epsilon>0$. Suppose that $\rank \varphi(x)=n$ for all
$x\in X$. Let $g=(g_1,\ldots, g_q)$ be functions, holomorphic
on $K$ such that $\varphi=dg$ on $K$.

 Then there exists a
continuous homotopy $H:X\times [0,1]\to (T^\ast X)^q$
such that
\begin{enumerate}
\item[(i)]
for each $t\in[0,1]$ are $(1,0)$-forms $H(\cdot,t)$ holomorphic on $K$, $H(\cdot,0)=\varphi$ and $H(\cdot,1)=df$
for some holomorphic functions $f=(f_1,\ldots,f_q):X\to \C^q$
satisfying $|f-g|_K<\epsilon$,
\item[(ii)] $\rank H(x,t)=n$ for all $(x,t)\in X\times [0,1]$.
\end{enumerate}
\end{theorem}

\begin{remark} {\it Case $n=1$.}

When $X$ is 1-dimensional Stein manifold (open Riemann
surface), the set $\Sigma=\{x\in X: \rank
(\varphi_1,\ldots,\varphi_{q-1})<1\}$ is a discrete set. By
\cite{F0}, Theorem 2.1, the problem $Lf|_\Sigma=g$ is solvable.
Therefore immersions $X\to \C^q$ satisfy homotopy principle for
all $q\geq 1$ (case $n=q=1$ is described in \cite{F0}, Theorem
2.5).

 {\it Case when $\dim X=q>1$ is still an open problem; method described in \cite{GR1} fails since one form is removed and the remaining forms should generate $T^\ast X$ somewhere.  However immersions $X\to \C^n$, where $X\subset \C^n$ is a contractible domain, satisfy homotopy principle and the obtained homotopy class consists of one element:}

Suppose $X=\C^n$.
Let $\varphi=(\varphi_1,\ldots,\varphi_n):
\C^n\to (T^\ast \C^n)^n$ be holomorphic $(1,0)$-forms with
$\rank \varphi=n$. There is a homotopy from $\varphi$ to
$df=(dz_1,\ldots, dz_{n-1}, dz_n)$, satisfying (ii) in the main theorem:

Let $A=(a_j^k)_{j,k}:\C^n\to \C^{n^2}$ be a (unique)
holomorphic map such that $\sum_{j=1}^n a_j^k \varphi_j = dz_j$
for $j=1,\ldots,n$. Let $G:\C^n\times
[0,1]\to {\text GL}_n(\C)\subset \C^{n^2}$ be a homotopy of
holomorphic maps from identity to $A$. A continuous homotopy
exists since $X$ is contractible and ${\text GL}_n(\C)$ connected.
By Oka principle \cite{F4} for sections of the bundle $(X\times \C^{n^2})\backslash \Sigma\to X$,
where fiber $\C^{n^2}\backslash \Sigma_x={\text GL}_n(\C)$ is homogeneous (hence it satisfies CAP), there exists a homotopy
$H':X\times [0,1]\to {\text GL}_n(\C)$ of holomorphic sections from $A$ to identity.
Set $H(x,t)=H'(x,t)\cdot \varphi(x)$.

{\it Question.} How to construct a homotopy when $X=\C^n\backslash \{0\}$?
\end{remark}

\begin{theorem}[parametric homotopy principle] \label{main_parametric}
Suppose $\varphi:X\times \Omega\to (T^\ast X)^q$ is a holomorphic map with $\rank \varphi(x,s)=n$
on $X\times \Omega$, where the parameter space $\Omega$ is a Stein manifold. Then there exists a continuous map $H:X\times \Omega \times [0,1]\to (T^\ast X)^q$, holomorphic in $(x,s)\in X\times \Omega$ with $H(x,s,0)=\varphi(x,s)$, $H(x,s,1)=df_s(x)$, and $H(\cdot,s,\cdot)$
satisfies (ii) in theorem \ref{main} for all $s\in \Omega$.
\end{theorem}

\begin{remark}
For applications the existence of continuous homotopy $\varphi:X\times \Omega\to T^\ast X$ satisfying the rank condition (ii) in
theorem \ref{main} is usually sufficient since it implies the existence of holomorphic
homotopy satisfying the same condition with additional interpolation and approximation of given $\varphi$:

 Homotopy is a section of the bundle $\pi \times {\text id}: (T^\ast X)^q\backslash A\times \Omega\to X\times \Omega$, where $A$ is an analytic variety describing $q$-tuples of forms with $rank<n$ . Fiber of the bundle is the complement of algebraic variety $\{\lambda \in \C^{n q}: \rank \lambda <n\}$ and has codimension $q-n+1\geq 2$, hence Oka principle for sections (theorem 1.1 in \cite{F4}) applies.
\end{remark}

\begin{corollary} \label{homot_immers}
Let $G:X\times [0,1] \to (T^\ast X)^q$ be a continuous homotopy of holomorphic forms satisfying (ii) in theorem \ref{main} and $G(\cdot,0)=df_1$, $G(\cdot,1)=df_2$, where $f_1,f_2:X\to \C^q$ are holomorphic
immersions. Then there exists a continuous homotopy $F:X\times [0,1]\to \C^q$ of holomorphic immersions such that
$F(\cdot,0)=f_1$, $F(\cdot,1)=f_2$.
\end{corollary}

\begin{proof}
We can construct continuous homotopy $G':X\times \C\to (T^\ast X)^q$, holomorphic near $\{0\}\cup \{1\}$, such that
$G'(\cdot,0)=G(\cdot,0)$, $G'(\cdot,1)=G(\cdot,1)$ and $\rank G'=n$ (define homotopy as a constant on
$\{{\text Re}\; t \in [-1/4,1/4]\} \cup \{{\text Re}\; t \in [3/4,5/4]\}\subset \C$). Note that $G'\times {\text id}$ with $\rank G'=n$
is a section of holomorphic bundle $\pi\times {\text id}: (T^\ast X)^q\backslash A\times \C\to X\times \C$ (see the above remark). By theorem 1.1 in \cite{F4} there exists a holomorphic section $G''$ such that $G''|_{\{0,1\}\times X}=G'$.

Now apply theorem \ref{main_parametric} with $\varphi=G''$ to obtain $H:X\times \C\times [0,1]\to (T^\ast X)^q$.
The path consisting of  the straight lines from  $(s,t)=(0,0)$ to $(0,1)$, then from $(0,1)$ to $(1,1)$
and lastly from $(1,1)$ to $(1,0)$, describes the homotopy $F$.
Note that all the perturbations in the proof of theorems \ref{main} and \ref{main_parametric} are exact.
\end{proof}

\section{Proof of main theorem}

\begin{proof}[Proof of Theorem \ref{main}:]
Final homotopy $H$ will be a conjunction of homotopies from steps (1)
and (2).

(1) {\bf Approximation with globally defined holomorphic
forms.}
 Using Oka principle (Lemma
\ref{approx_OKA}) construct a homotopy $H_0$ satisfying (ii)
such that $H_0(\cdot,0)=\varphi$ and $H_0(\cdot,1)$ is
holomorphic on $X$ with $|H_0(\cdot,1)-\varphi|_K<\epsilon_1$.
Here $\epsilon_1$ is chosen small enough (more precisely $(M+1)^q
\epsilon_1<\delta$ and $(M+1)^q \epsilon_1<\epsilon$ where
$M=M(g|_K)$ and $\delta=\delta(g|_K)$ are from theorem
\ref{solve_lf_a}; see also step (2)) such that the
approximations
in step (2) will be small enough. Note that exactness
of the forms on $K$ is not preserved in this step.

(2) {\bf Replacing forms one by one with differentials.} For
each $k=1,\ldots,q$ we inductively construct homotopies $H_k$
satisfying (ii) such that forms in
 $H_k(\cdot,1)$ $M \epsilon_1 (1+M)^{k-1}$-approximate forms in $H_{k-1}(\cdot,1)=H_k(\cdot,0)$ on $K$, where $H_0(\cdot,1)=\varphi$.
Since $\epsilon_1 (1+M)^{k-1}+M \epsilon_1
(1+M)^{k-1}=\epsilon_1 (1+M)^k$, forms in $H_k(\cdot,1)$ will
$\epsilon_1 (1+M)^k$-approximate forms in original $\varphi$.
Each homotopy $H_k$ will replace $k$-th form by a diferential
of holomorphic function.

\noindent-(2.1) {\bf Obtaining semi-transversality.}
 Suppose that at the beginning of $k$-th step we
have the forms
\[(\varphi_1=dh_1, \ldots, \varphi_{k-1}=dh_{k-1},\varphi_k, \ldots,\varphi_q)=H_{k-1}(\cdot,1).\]
First we use the homotopy $H_k'$ to generically perturb forms
$\varphi_j$, $j\neq k$ in order to obtain semi-transversality
condition $L\pitchfork_{\text{semi}} \Sigma$ (lemmas
\ref{nontrans_base}, \ref{nontrans_reduct} and proposition \ref{hom_semi_trans}) needed
in step (2.2). Here $\Sigma=\Sigma(\varphi)=\{x\in X:\rank
(\varphi_j(x))_{j\neq k}<n\}$, more precisely
$\Sigma=\cap_{\{\alpha_1,\ldots,\alpha_n\}\subset
\{1,\ldots,q\}\backslash  \{k\}} \{ x\in X:
\rank(\varphi_{\alpha_1},\ldots,\varphi_{\alpha_n})<n\}$. $L:X\to
TX$ is a holomorphic vector field(uniquely defined on $\Sigma$ since
$\rank \varphi=n$), satisfiying $\varphi_j L=0$ for all $j\neq k$
and $\varphi_k L=1$ on $\Sigma$. Globally defined $L$ is
obtained by extension using Cartan's theorems.

-(2.2) {\bf Solving $L f_k|_{\Sigma} = 1$ with approximation on
$K$.} Find the holomorphic function $f_k=g_k+h:X\to \C$ where $h$
satisfies $Lh|_\Sigma=1-Lg_k|_\Sigma$. Since perturbation of $L|_K$
remains bounded when $\varphi$ is perturbed a little on $K$
($L\mapsto L|_\Sigma: \Gamma(X,TX)\to \Gamma(\Sigma,TX)$ is
continuous, linear surjection, hence open by open mapping theorem)
and $\Sigma\cap K$ is perturbed a little if $\varphi|_K$ is
perturbed a little, $1-Lg_k=1-dg_k L=\varphi_k L-dg_k L$ is close to
$0$ on $K$, hence there exists a solution $h$ which is $M |1-dg_k
L|_K$-close to $0$ on $K$ by Theorem \ref{solve_lf_a}. Note that
functions $f_j$ (and $\tilde f_j$) in theorem \ref{solve_lf_a} are
in our case the defining equations for the set $\Sigma$, described
in detail in the proof of Lemma \ref{nontrans_base}.
 Hence $f_k$ approximates $g_k$
on $K$. Define the homotopy as \[ H_k''(\cdot,t)=(\varphi_1,\ldots,
\varphi_{k-1},(1-t)\varphi_q+t d f_k,\varphi_{k+1},\ldots,
\varphi_q).\] Note that $H_k''$ satisfies (ii) since $\varphi_q
L=df_k L=1$. $H_k$ is a conjuction of homotopies $H'_k$ and
$H''_k$, that is $H_k(\cdot,t)=H_k'(\cdot,2t)$ if $0\leq t\leq
1/2$ and $H_k(\cdot,t)=H_k''(H_k'(\cdot,1),2t-1)$ if $1/2\leq
t\leq 1$.
\end{proof}

\begin{lemma} \label{approx_OKA}
Given forms $\varphi$ as in the main theorem \ref{main} there
is a continuous  homotopy $H:X\times [0,1]\to (T^\ast X)^q$ of
forms holomorphic on $K$ and satisfying (ii) in main theorem,
such that $H(\cdot,1)$ is holomorphic on $X$ and approximates
$\varphi$ on $K$.
\end{lemma}

\begin{proof}
There are finitely many holomorphic forms $\Phi_1,\ldots,\Phi_N$
$(N\geq n)$ with rank $n$ on $X$. Hence we can write
$\varphi_k=\sum_{j=1}^N a_j^k \Phi_j$ for some functions $a_j^k$,
holomorphic on $K$. $({\text id},a_j^k):X\to X\times \C^{Nq}$ is a
section of the holomorphic fiber bundle $\pi:(X\times \C^{Nq})\backslash A\to X$
where $A=\{(x,a)\in X\times \C^{Nq}: \rank
(\varphi_1,\ldots,\varphi_q)\leq n-1\}\subset X\times \C^{Nq}$ is
analytic variety and $\pi:X\times \C^q\to X$ is a standard
projection. Since the complements of fibers $A_x=A\cap (\{\{x\}\times \C^q)$ are
affine hyperplanes of codimension equal to $q-(n-1)\geq 2$, we can
appy theorem 1.1(B) in \cite{F4} to the sections of the bundle
$\pi:(X\times \C^{Nq})\backslash A\to X$  and obtain the desired
homotopy.
\end{proof}

\begin{proof}[Proof of theorem \ref{main_parametric}] We follow the proof of theorem \ref{main} and check that all the steps can be done holomorphically in parameter  $s$. Define $\tilde \Sigma=\{ (x,s): x\in \Sigma_s\}$, where $\Sigma_s=\{x\in X:
\rank \varphi(x,s)<n\}$. We work with family of vector fields $\tilde L:X\times \Omega\to TX$, which is uniquely defined on $\tilde \Sigma$; global extension is obtained using Cartan's theorems.

Step (2.1):
 When obtaining semi-transversality in step (2.1)
define \[\tilde \Sigma''=\{(x,s)\in \tilde \Sigma:
(L_{U})^k h_{j_1,\ldots, j_n}^l (x)=0 \;{\rm for\; all\;} L_{U}, l, 1\leq k \leq n+1\} .\]
The goal, to reduce $\tilde \Sigma''$ to $\emptyset$, is achieved as before.
Observe that $df$ in lemma \ref{nontrans_base} can be choosen with holomorphic dependence on $s$. $f$
is a generic map, obtained from the jet transversality theorem, such that the jet of $f$ is transversal to some
variety $\Lambda=\Lambda(s)$ in the jet space. Suppose $\Lambda(s)$ is defined by $g(x,s,\lambda_x)=0$, where
$(x,\lambda_x)\in J^{m+1} (X,\C)$.
Define $\tilde \Lambda=\{(x,s,\lambda_x,\lambda_s)\in J^{m+1} (X\times \Omega,\C): g(x,s,\lambda_x)=0\}$ and note that
$\tilde \Lambda$ is an analytic variety with $\text codim\; \tilde \Lambda\geq \text codim\; \Lambda$. The jet of generic $F:X\times \Omega\to \C$ is transversal to $\tilde \Lambda$.
Hence when $\text codim \; \Sigma(s)=m> \dim (X\times \C)=n+1$, $j_x^{m+1} f_s$ will miss $\Sigma(s)$ for all $s$.
Hence a generic perturbation of forms in step (2.1) can be constructed of the form $\varphi+df_s$, where $f_s=F(\cdot,s)$
depends holomorphically on $s$.

Step (2.2) To get a solution $f_s$ (holomorphically depending on parameter $s$) to the problem $L_s f|_{\Sigma_s}=1$
use theorem \ref{solve_lf_a_param}.
\end{proof}

\section{Solving $Lf|_\Sigma=g$ with approximation on compacts and with parameters}

Let $X$ be a Stein manifold, $\Sigma\subset X$ an analytic set
and $L:X\to TX$ a nowhere on $\Sigma$ vanishing holomorphic
vector field. Let $\mathcal J(\Sigma)$ be the sheaf of ideals
in $\cO_X$, consisting of the germs of holomorphic functions
that vanish on $\Sigma$.

\begin{example}
Solution $f$ to the problem $Lf|_\Sigma=g$ does not always
exist. Suppose that (extreme opposite of the assumptions in Theorem
\ref{solve_lf_a}) $L$ is tangential to $\Sigma$ to infinite order at all points of $\Sigma$. Then
$f'=f|_\Sigma$ is a solution of the problem $Lf'=g$, where $f'$
is holomorphic on $\Sigma$.

The simplest example where there is no solution $f$ is $L={\di\over \di z}$, $g=1$ and $X=\{(z,w)\in \C^2: z\neq 0\}$.
More complicated example \cite{Wak} is a domain $\Sigma \in \C^3$,
biholomorphic to a polydisc, where the problem ${\di f'\over
\di z_1}=g$ is not solvable for some holomorphic functions $g$.
In this example there is a complex line $L=\{z_2=const., z_3=const.\} \subset
\C^3$, such that $L\backslash (L\cap \Sigma)$ has bounded
components. Choose a point $(z^0_1,z_2,z_3)$ in such a bounded
component and let $g(z)$ be holomorphic extension of
$1/(z_1-z^0_1)$ from $L\cap D$. Also note that a generic function has a nonzero residue,
hence $Lf=g$ is not solvable for most functions $g$.
 The problem is not solvable
even if $\Sigma$ is Runge domain in $\C^n$\cite{Wak}.

But in the case when there exists a solution $f'$, the solution
to $Lf|_\Sigma=g$ is any holomorphic extension $f$ of $f'$,
since $L(f-f')=0$ on $\Sigma$ (tangentiality).
\end{example}

\begin{question}
1. Can we solve the problem $Lf|_\Sigma=1$ if we know how to solve the problem $Lf|_\Sigma\neq 0$ ?\\
2. Can we solve the problem $Lf\neq 0$ on a sphere $S\subset \C^n$ such that $f$ approximates given $g$ on a compact $K\subset \C^n$ provided that $Lg\neq 0$ on $K$ ?
\end{question}

The following theorem shows that the solution $f$ to the problem $Lf|_\Sigma=g$ changes a little on a compact $K$
if the input data $L,g$ are changed a little on $K$.

\begin{theorem} \label{solve_lf_a}
Let $K\subset X$ be a compact, holomorphically convex set.
Suppose there exist holomorphic functions $f_1,\ldots,f_N\in
\mathcal J(\Sigma)$ such that
\begin{equation} \label{eq_generating1}
\{f_1=0,\ldots, L^N f_1=0, \ldots, f_N=0,\ldots, L^N f_N=0\}=\emptyset.
\end{equation}

 There exists
$\delta>0$ and $M=M(K,f|_K,L|_K)>0$, such that given analytic
set $\tilde \Sigma\subset X$,holomorphic vector field $\tilde L$ with
$|\tilde L-L|_K<\delta$, holomorphic functions $\tilde f_j\in
\mathcal J(\tilde \Sigma)$ satisfying $\sum_{j=1}^N |\tilde
f_j-f_j|_K<\delta$ and (\ref{eq_generating1}) and a holomorphic
function $\tilde g:X\to \C$, we have $\tilde L\tilde f|_{\tilde
\Sigma}=\tilde g|_{\tilde \Sigma}$ for some holomorphic
function $\tilde f:X\to \C$ satisfying $|\tilde f|_K\leq M
|\tilde{g}|_K$.
\end{theorem}

\begin{proof}
Condition (\ref{eq_generating1}) implies the existence of
holomorphic functions $a_j^1,\ldots,a_j^N:X\to \C$
($j=1,\ldots,N$) such that
\begin{equation} \label{eq_generating2} \sum_{j=1}^N a_j^1 L f_j+\ldots +\sum_{j=1}^N a_j^N L^N f_j=1
\end{equation}
on $\Sigma$, hence $\sum_{j=1^N} g a_j^1 L f_j+\ldots
+\sum_{j=1^N} g a_j^N L^N f_j=g$. The solution to the problem
$Lf|_\Sigma=g$ is then
\begin{equation} \label{formula_main}
 f= \sum_{j=1}^N ga_j^1 f_j+\ldots + \sum_{j=1}^N ga_j^N L^{N-1} f_j-
 \end{equation}
\[-(\sum_{j=1}^N L(ga_j^2) f_j+\ldots+\sum_{j=1}^N L(ga_j^N) L^{N-2} f_j)-\]
\[ -\ldots - (\sum_{j=1}^N L^{N-1}(ga_j^N) f_j)\]

The proof is completed  by applying the following lemma to the
coefficients $a_j^k$ in equation
(\ref{eq_generating2}).\end{proof}

\begin{remark}

(1) Since $L$ is semi-transversal to $\Sigma$, there are (Lemma
\ref {semitransvers1}) finitely many holomorphic functions
$f_1,\ldots,f_N$ with $f_j\in \mathcal J(\Sigma)$ satisfying
(\ref{eq_generating1}). Hence the problem $Lf|_\Sigma=g$ has a
solution. But in our application we need to show that the sup
norm $|f|_K$ of the solution is small whenever $|g|_K$ is
small. Moreover, $\Sigma$ changes during the proof of the main
theorem.

(2) Let $\Lambda\subset X$ be an analytic set. If $g\in \mathcal J(\Lambda)^{N}$ then we can conclude $f\in \mathcal J(\Lambda)$. Can we obtain such $f$ when $g\in \mathcal J(\Lambda)$?
\end{remark}

\begin{lemma}
Let $K\subset X$ be a compact, holomorphically convex sets and
let $f=(f_1,\ldots,f_N):X\to \C^N$ be holomorphic functions without
common zero and let $a=(a_1,\ldots,a_N):X\to \C^N$ be holomorphic
functions such that $a.f=\sum_{j=1}^N a_j f_j=1$. Set $\delta=1/(2|a|_K)>0$.

For every $\tilde f=(\tilde f_1,\ldots,\tilde f_N)$ without
common zeroes such that $|\tilde f-f|_K= \sum_{j=1}^N |\tilde
f_j-f_j|_K<\delta$ there exist functions $\tilde a=(\tilde
a_1,\ldots,\tilde a_N):X\to \C^N$ with $\tilde a.\tilde f=1$
and $|\tilde a-a|_K<4|a|_K^2 |f-\tilde f|_K$.
\end{lemma}

\begin{proof}
Let $a_0=a/a.\tilde f$. Now $|a_0-a|_K\leq 2|a|_K |1-a.\tilde
f|_K\leq 2 |a|_K^2 |f-\tilde f|_K$. Let $a':X\to \C$ be such
that $a'.\tilde f=1$. Since $a_0-a':K\to X$ is a section of
coherent analytic sheaf $\cF=\{a\in \cO^N_X: a.\tilde f=0\}$,
we can approximate it on $K$ with a global section $a''\in
\Gamma(X,\cF)$ such that $|a_0-a'-a''|_K<2 |a|_K^2 |f-\tilde
f|_K$. Set $\tilde a=a'+a''$.
\end{proof}

The proof of Theorem \ref{solve_lf_a} also gives the following parametric version:

\begin{theorem} \label{solve_lf_a_param}
Let $\Omega\subset \C$ be a Stein manifold.
Let $G:X\times \Omega\to \C$ be a holomorphic function and let
$\tilde \Sigma\subset X\times \Omega$ be an analytic set.
Let $L:X\times \Omega\to TX$ be a holomorphic
map such that $L_s=L(\cdot,s):X\to TX$ is a holomorphic vector field for all $s\in \Omega$.

Set $\Sigma_s=\{x\in X: (x,s)\in \tilde \Sigma\}$ and $g_s=G(\cdot,s)$.
Suppose that $L_s\pitchfork_{\text{semi}} \Sigma_s$ for all $s\in \Omega$. Then
there is a holomorphic function $F:X\times \Omega\to \C$, such that
$f_s=F(\cdot,s)$ is a solution to the problem
$Lf|_{\Sigma_s}=g_s|_{\Sigma}$ for all $s\in \Omega$.
\end{theorem}

\begin{proof} Use the (proof of) lemma \ref{semitransvers1} to choose a number $N\in \N$
and holomorphic functions $F_1,\ldots, F_N:X\times \Omega\to \C$ in $\mathcal J(\tilde \Sigma)$
such that
$\{ (x,s): F_1=0, L_s^N F_1=0,\ldots, L_s^N F_N=0\}=\emptyset$.
Now choose holomorphic functions $A_j^l:X\times \Omega\to \C$ as in
(\ref{eq_generating2}) and write the solution $F$ as in the proof of theorem \ref{solve_lf_a}.
\end{proof}

\section{Semi-transversality}

\begin{definition} \label{def_semitrans}
Let $L:X\to TX$ be a holomorphic vector field on a complex
manifold $X$ of dimension $n$ and let $\Sigma\subset X$ be an analytic set. $L$
is {\it tangent to $\Sigma$ with order less or equal $k$} at
$x\in \Sigma$, if there is a holomorphic function $f\in
\mathcal J(\Sigma)=\{g\in \cO_X: g|_\Sigma=0\}$ with $L^j_x f= L(\ldots L(f)\ldots )_x \neq 0$ for some $j\leq k$. We
say that $L$ is {\it semi-transversal} (see \cite{GR1}) to
analytic variety $\Sigma\subset X$ and write
$L\pitchfork_{\text{semi}} \Sigma$, if $L$ is tangent to
$\Sigma$ with finite order at each point of $\Sigma$.
\end{definition}

The following lemma explains why it is sufficient to check semi-transversality condition
for local extensions of vector field $L|_\Sigma$.

\begin{lemma} \label{semitransvers1}
Supppose that each $x\in \Sigma$ has a neighborhood
$U$ and a holomorphic vector field $L_U:U\to TX$ with
$(L_U)|_{\Sigma\cap U}=L|_{\Sigma\cap U}$, such that $L_U$ is
tangent to $\Sigma$ with order less or equal $N$ at each point of
$U\cap \Sigma$. Then there is a finite number
$N'\leq n+1$ of functions $f_1,\ldots,f_{N'}\in \mathcal
J(\Sigma)$ such that $\{ f_1=0,\ldots, L^{N} f_1=0,\ldots,
f_{N'}=0,\ldots , L^N f_{N'}=0\}=\emptyset$. In other words, $L$
is tangent to $\Sigma$ with the order less than $N$ at each
point of $\Sigma$.
\end{lemma}

\begin{proof} Also see lemma (\textsection{2.1.5}, p.66) in \cite{GR1}.
Inductively define a sequence of sheafs of ideals by
$I_0=\mathcal J(\Sigma)$ and $I_{k+1}=I_{k}+\cO_X\cdot \{L f: f\in I_k\}$
for $k\geq 1$. Sheafs $I_k$  satisfy $I_k\subseteq I_{k+1}$ for
all $k$. Note that replacing $L$ by $L'$ such that $L-L'=0$ on
$\Sigma\cap U$ does not change ideal $I_{k+1}$ over $U$ since
$\mathcal J(\Sigma)\subset I_k$. Hence the ideals are dependent
only of $L|_{\Sigma}$. Suppose $L_U$ is tangent to $\Sigma$
with order less or equal $k$ on $U\cap \Sigma$. Then $I_{k}=\cO_X$
over $U$. This implies the existence of $N'\in \N$ and
functions $f_1,\ldots, f_{N'}$:

Choose $x\in \Sigma$ and a function $f_1\in \mathcal J(\Sigma)$
with $L^{N_1}_x f_1\neq 0$ for some $N_1\leq N$. Such function exists
because $I_{N_1}=\cO_X$ on a neighborhood of $x$.
 Therefore $\dim
\{f_1=0,Lf_1=0,\ldots, L^{N_1} f_1=0\}<n=\dim X$. Suppose we
have constructed functions $f_1,\ldots, f_k\in \mathcal
J(\Sigma)$ and numbers $N_1\leq \ldots \leq N_k\in \N$ such
that the dimension of analytic set $\Sigma^k_1=\{f_1=0,\ldots,
L^{N_k} f_k=0\}$ is less or equal $n-k$. We can assume that $\Sigma^k_1\subset \Sigma$;
if not just add defining functions for $\Sigma$.
Choose a point $x_j$
in each $(n-k)$-dimensional irreducible component of
$\Sigma^k_1$. Proceed as in the proof of Proposition (\textsection{5.7}) in
\cite{C}, which proves that analytic sets in complex manifolds
can be defined by finitely many equations, by selecting the
function $f_{k+1}\in \mathcal J(\Sigma)$ such that $x_j\not \in
\Sigma^{k+1}_1=\{f_1=0,\ldots, L^{N_{k+1}} f_{k+1}=0\}$ for all
$j$. Such function is of the form $f_{k+1}=\sum_{l\in \N} c_l f^l_{k+1}$, where $L^{N_l} f^l_{k+1} (x_l)\neq 0$ for some $N_l\leq N$ and all $l\in \N$.
Functions $f^l_{k+1}\in \mathcal J (\Sigma)$ exist since $I_{N_l}=\cO_X$ on a neighborhood of $x_l$.
Numbers $c_l$ are chosen inductively on $l$ such that $L^{N_l} f_{k+1} (x_l)\neq 0$ for all $l\in \N$ and such that the sum converges uniformly on compacts in $X$.
Hence $\dim \Sigma^{k+1}_1<\dim \Sigma^k_1$. Therefore in
$N'\leq n+1$ many steps we get $\Sigma^{N'}_1=\emptyset$. Set
$N=N_{N'}$.

\end{proof}

Now we describe how to obtain semi-transversality condition
$L\pitchfork_{\text{semi}} \Sigma$ needed in the step (2.2) of
the proof of the main theorem. All notation is as in step (2)
of the proof. We will work with $k=q$. Let \[\Sigma=\cap_{1\leq
j_1 < \ldots j_n<q}
\Sigma(\varphi_{j_1},\ldots,\varphi_{j_n})\] where
$\Sigma(\varphi_{j_1},\ldots,\varphi_{j_n})=\{x\in X:
\rank(\varphi_{j_1},\ldots,\varphi_{j_n})<n\}$. Let
$h_{j_1,\ldots, j_n}^l$ be the defining equations for each set
of the intersection. Let $L:X\to TX$ be an extension (existence
follows from Cartan theorems) of holomorphic field $\Sigma\to
TX$, uniquely defined  on $\Sigma$ by $\varphi_1 L=0,\ldots,
\varphi_{q-1} L=0$, $\varphi_q L=1$.

{\bf Local extensions $L_U$} of vector field $L|_\Sigma$:

$\rank(\varphi_1,\ldots,\varphi_q)=n$ on $X$, hence for each
$x\in \Sigma$ there are $n-1$ ($\varphi_2,\ldots,\varphi_{n}$
for sake of notation) forms among
$\varphi_1,\ldots,\varphi_{q-1}$ which together with
$\varphi_q$ span $T^\ast V$ on a neighborhood $U=V\backslash
\Sigma(\varphi_2,\ldots,\varphi_{n},\varphi_q)$ of $v$. Every
choice of such $(n-1)$-tuple of forms defines the extension
$L_U=L(\varphi_2,\ldots,\varphi_n):U\to TX$ of the vector field
$L|_{\Sigma}$ defined by $\varphi_1
(L|_\Sigma)=0,\ldots,\varphi_{q-1}(L|_\Sigma)=0$ and $\varphi_q
(L|_\Sigma)=1$. There are at most $\binom{q-1}{n-1}$ possible
choices.
 Let $\Sigma''$ be defined as a finite
intersection
\[\Sigma''=\{x\in \Sigma:
(L_U)^k h_{j_1,\ldots, j_n}^l (x)=0 \;{\rm for\; all\;} L_U, l, 1\leq k \leq n+1\} ,\]

In other words $L_U$ is tangent to $\Sigma$ with the order less
than $n$ at all points $v\in U\cap (\Sigma\backslash\Sigma'')$.
If $\Sigma''=\emptyset$ then for each $x\in \Sigma$ there is a
local extension $L_U$ tangent with finite order to $\Sigma$ on
$U$, $x\in U$. Lemma \ref{semitransvers1} then shows that
$L\pitchfork_{\text{semi}} \Sigma$.

 We will reduce $\Sigma''$ to $\emptyset$ in countably many
steps. First we show how to establish upper bound on the order
of tangentiality outside some analytic set $\Sigma'$.

\begin{lemma} \label{nontrans_base}
Choose $x\in \Sigma$, a compact $K\subset X$, $k\in \N$ and
$\epsilon>0$. Suppose that $x\not\in
\Sigma'=\Sigma(\varphi_2,\ldots,\varphi_n,\varphi_q)$ (for each
$x\in \Sigma$ there are $n-1$ forms that span $T^\ast V$
together with $\varphi_q$). Set $U=X\backslash \Sigma'$.

Then there exists a holomorphic function $f\in {\mathcal
J}(\Sigma')^k$ such that the vector field
$L_U=L(\varphi_1+df,\varphi_2,\ldots,\varphi_{q-1},\varphi_q)$
is tangential to
$\Sigma(\varphi_1+df,\varphi_{2},\ldots,\varphi_{q-1})$ with
the order less than $n+1$ on $\Sigma \backslash \Sigma'$ and
$df$ is $\epsilon$-close to $0$ on $K$.
\end{lemma}

\begin{proof}
Such $f$ will be obtained from jet-transversality theorem for
holomorphic functions on a Stein manifold. Note that
$(\cap_{j=1}^{q-1} \ker \varphi_j)|_{\Sigma\backslash \Sigma'}=
(\cap_{m=2}^{n} \ker \varphi_{m})|_{\Sigma\backslash \Sigma'}$
is independent of $\varphi_{1}$. Hence by changing $\varphi_1$
we do not change $L$ on $\Sigma\backslash \Sigma'$.
 Let $U$ be a neighborhood of
an arbitrarily choosen point in $X\backslash \Sigma'$, such
that in the local coordinates on complex manifold $X$ we have
$L={\di \over \di x_1}$.

The defining equation for
$\Sigma(\varphi_1+df,\varphi_{2},\ldots,\varphi_{n})$ on a
neighborhood $U$ of point $x$ is $\det
[\varphi_1+df,\varphi_{2},\ldots,\varphi_{n}]=0$. By expanding
the determinant of $n\times n$ matrix \[A=[\varphi_1+df,\ldots,
\varphi_{n}]\] by the first row we get
\begin{equation}
h(f,v)=\det A=\sum_{j=1}^n ({df\over dx_j}+\varphi_1^j)(x) A_j(x)=0,
\end{equation}
where $A_j$ is $(1,j)$-minor (depending only on
$\varphi_2,\ldots,\varphi_{n}$). The tangentiality condition in
Lemma \ref{nontrans_base} will be satisfied if for  $m=n+1$ the
vector field ${\di \over \di x_1}$ is tangential to $\Sigma$
with the order at most $m$ on U. We need the function $f$ such
that
\[ Lh={\di h(f,x)\over \di x_1}=0 \ldots, L^m h={\di^{m} h(f,x)\over \di x_1^{m}}=0 \]
is true nowhere on $X\backslash \Sigma'$. This is equivalent to
$j^{m+1} f(X\backslash \Sigma')\cap \Lambda=0$, where
$\Lambda\subset J^{m+1}(X\backslash \Sigma',\C)$ is an analytic
set. By observing the definition of $h(f,x)$ we see that ${\rm
codim}\; \Lambda=m>n$ if $m>n$ (note that at least one minor
$A_j$ is nonzero at each point). By jet transversality theorem
for holomorphic functions $X\backslash \Sigma' \to \C$ the jet
of a generic holomorphic function $f:X\backslash \Sigma'\to \C$
(we can choose $f$ to be close to $0$ on $K$) is transversal to
$\Lambda$ on $V\backslash \Sigma'$. To complete proof choose
exhaustion of $X\backslash \Sigma'$ by compacts $K_j$ and
construct a sequence of uniformly on compacts in $X\backslash
\Sigma'$ converging generic holomorphic maps $f_j\in\mathcal
J^k(\Sigma)$, such that the jet of each map is transversal to
$\Lambda$ on a compact $K_j$ in $V\backslash \Sigma'$ and $f_j$
approximates $f_{j-1}$ on $K_{j-1}$. Set $f=\lim_{j\to \infty}
f_j$.

\end{proof}

This describes the reduction of the set $\Sigma''$ of
"bad" points in $\Sigma$.

\begin{lemma} \label{nontrans_reduct}
Notation is from the previous lemma. Denote by
$\Sigma''(\varphi_1)$ the dependence of $\Sigma''$ on $\varphi_1$.
There is a holomorphic function $f\in \mathcal J(\Sigma')^{n+2}$
with $|f|_K<\epsilon$ such that:\\
(1) $\tilde L_U=L(\varphi_1+df,\varphi_2,\ldots,\varphi_{q-1})$
is tangential to $\Sigma$ with the order less than $n+1$
on $\Sigma\backslash \Sigma'$,\\
(2) $\Sigma''(\varphi_1+ df)=\Sigma''(\varphi_1)\cap \Sigma'$,\\
(3) The homotopy $t\mapsto (\varphi_{1}+t df, \varphi_2, \ldots,\varphi_{q-1},\varphi_q)$ satisfies (ii) in theorem \ref{main}.
\end{lemma}

\begin{proof}
First part follows from $\Sigma\subset \Sigma(\varphi_1+df,\varphi_2,\ldots,\varphi_{n})$
and the previous lemma.

The second part follows from $\Sigma''(\varphi_1+df)\subset
\Sigma\cap \Sigma'$ and the following. Suppose $x\in
(\Sigma\cap \Sigma')\backslash \Sigma''(\varphi_1)$. By
definition of $\Sigma''(\varphi_1)$ the order of tangency of
$L|_U$ to $\Sigma$ near $x$ is less or equal $n$. The difference
between the old and the new defining function for $\Sigma$ near $x$ lies
in $\mathcal J( \Sigma')^{n+1}$ since $f\in J( \Sigma')^{n+2}$.
Closer examination of the linear system defining $L|_U$ shows
the same is true for the change of $L|_U$. Then the order of
tangentiality of $L|_U$  to $\Sigma$ at $x$ is preserved; this
is a simple consequence of the definition of the order of
tangency and of the derivation rule for products.

The last part follows from the definition of
$\Sigma'=\Sigma(\varphi_2,\ldots,\varphi_n,\varphi_q)$ and the fact
that $df=0$ on $\Sigma'$.
\end{proof}

\begin{proposition} \label{hom_semi_trans}

Let $K\subset X$ be a compact set and let $\epsilon>0$.
There is a continuous homotopy $\Psi:X\times [0,1]\to (T^\ast X)^{q-1}$ such that\\
(1) $\Psi(0)=(\varphi_1,\ldots,\varphi_{q-1})$,\\
(2) for every $x\in \Sigma$ the corresponding vector field
$L_U=L(\Psi(1))$ is tangential to $\Sigma$ with order at most
$n+1$ on a neighborhood $U$ of $x$\\
(3) Homotopy $(\Psi(\cdot, t),\varphi_q)$ satisfies (ii) in
theorem \ref{main} and $\Psi(t)$ is $\epsilon$-close to
$\Psi(0)$ on $K$ for all $t\in [0,1]$
\end{proposition}

\begin{proof}
Homotopy is obtained by consecutively joining countably many
homotopies described by Lemma \ref{nontrans_reduct}. We choose
an exhaustion of Stein manifold $X$ with compacts. At each step
the forms are modified by homotopy (Lemma
\ref{nontrans_reduct}) and we make sure the sequence of
modified forms  converges uniformly on compacts (lemma
\ref{nontrans_reduct}).

The goal is to reduce the "bad" set $\Sigma''$ of points,
tangential with infinite order, to $\emptyset$. At each step
$\Sigma''$ is contained in some analytic variety. We show that
at each step the number of irreducible components of highest
dimension of this variety is reduced:

Let  $\Sigma''_1=\Sigma''\subset \Sigma_1=\Sigma$ and let
$\Sigma^1$ be the union of all irreducible components of
$\Sigma_1$ having nonempty intersection with $\Sigma''_1$. We
have $\Sigma_1''\subset \Sigma^1$. Now choose $x\in \Sigma_1''$
lying in highest dimensional
    irreducible component of $\Sigma^1$. Note that
$\Sigma'_1=\Sigma'$ in the proof of lemma \ref{nontrans_base}
is choosen such that some neighborhood of $x$ in $X$ is
disjoint to $\Sigma'_1$. Now $\Sigma''_2$ (new $\Sigma''$) is
equal to $\Sigma''_1\cap \Sigma'_1\subset \Sigma^1\cap
\Sigma'_1$. But $\Sigma^2=\Sigma^1\cap \Sigma'_1$ is missing at
least one of highest dimensional irreducible component of
$\Sigma^1$. Since analytic variety can have at most countably
many highest dimensional irreducible components, we get
$\Sigma''\subset \emptyset$ in countably many steps.
\end{proof}

\begin{remark}
Countably many steps are needed in the reduction in our proof.
That is reason why the final homotopy is only piecewise
smooth(countably many pieces), since it is obtained as a
conjunction of countably many homotopies.
\end{remark}

\end{document}